\documentclass[paper=a4,english,fontsize=11pt,parskip=half,abstract=true]{scrartcl}
\usepackage{babel}
\usepackage[utf8]{inputenc}
\usepackage[T1]{fontenc}
\usepackage[left=20mm,right=20mm,top=30mm,bottom=30mm]{geometry}
\usepackage{amsmath}
\usepackage{amsthm}
\usepackage{amssymb}
\usepackage{enumerate} %superseeded by enumitem
\usepackage{thmtools}
\usepackage{mathtools}
\usepackage{multicol}
\mathtoolsset{centercolon} %makes := symmetric
\usepackage[bookmarks=true,
            pdftitle={Character table sudokus},
            pdfauthor={Benjamin Sambale},
            pdfkeywords={},
            pdfstartview={FitH}]{hyperref}

\newtheorem{Thm}{Theorem} %[section]
\newtheorem{Lem}[Thm]{Lemma}

\newtheorem{Prob}[Thm]{Problem}

\theoremstyle{definition}

\numberwithin{equation}{section}

\allowdisplaybreaks[1]

\renewcommand{\phi}{\varphi}
\newcommand{\C}{\mathrm{C}}

\newcommand{\Z}{\mathrm{Z}}
\newcommand{\pcore}{\mathrm{O}}

\newcommand{\CC}{\mathbb{C}}

\newcommand{\Aut}{\mathrm{Aut}}

\newcommand{\Out}{\mathrm{Out}}
\newcommand{\GL}{\mathrm{GL}}

\newcommand{\Irr}{\operatorname{Irr}}

\newcommand{\Ker}{\operatorname{Ker}}

\newcommand{\TT}{\mathrm{t}}
\newcommand{\ii}{\mathrm{i}}

\title{Character table sudokus}
\author{Benjamin Sambale\footnote{Institut für Algebra, Zahlentheorie und Diskrete Mathematik, Leibniz Universität Hannover, Welfengarten 1, 30167 Hannover, Germany,
\href{mailto:sambale@math.uni-hannover.de}{sambale@math.uni-hannover.de}}}
\date{\today}

\begin{document}
\frenchspacing
\maketitle
\begin{abstract}\noindent
It is a fun game to complete a partial character table of a finite group. We show that one can reconstruct a missing row or column from a given table. The proof relies on deep properties of fully ramified characters. Moreover, we extend a classification of groups with a “large” character degree started by Snyder and continued by Durfee and Jensen.
\end{abstract}


\renewcommand{\sectionautorefname}{Section}
\section{Introduction}

Character tables of finite groups are complex square matrices satisfying a large number of arithmetical properties, most prominently, the orthogonality relations (see \cite{GagolaFormal} for a compilation). Filling in missing values of a partial character table, can therefore be seen as a sudoku-like puzzle. 
A particular challenge presents itself when an entire row or column of a given table is vacant. We show that in both cases one can reconstruct uniquely the missing row or column by using only the given part of the table (i.\,e. without using other properties of the underlying group). 

\begin{Thm}\label{main}
There are no finite groups whose character tables differ by only one row or only one column.
\end{Thm}

As usual we consider character tables as identical if they only differ by permuting rows and columns. 
Our proof of \autoref{main} provides an explicit algorithm and we challenge the reader with some examples. 
Playing this game further will sooner or later lead to so-called \emph{pseudo groups} introduced by Brauer~\cite{BrauerPseudo} and investigated further by Harris~\cite{HarrisPseudo} and Gagola~\cite{GagolaFormal} (a concrete example is given in the next section). In his Problem~6, Brauer~\cite{BrauerLectures} has even asked to give necessary and sufficient conditions distinguishing character tables from arbitrary matrices.
We put forward the following open problem.

\begin{Prob}
Do there exist distinct character tables which differ by at most one entry in every row \textup(or in every column\textup)?
\end{Prob}

In the course of the paper we need to revisit groups $G$ with characters of “large” degree (compared to $|G|$). Such groups were studied and classified by Snyder~\cite{Snyder} and Durfee--Jensen~\cite{DurfeeJensen}. In the last section we use the opportunity to extend their classification. 

\section{Proof of \autoref{main}}

We split up \autoref{main} and start with the easier case of a missing column.

\begin{Thm}\label{thmcols}
There are no finite groups whose character tables differ by only one column.
\end{Thm}
\begin{proof}
Let $C=(c_{ij})\in\CC^{k\times (k-1)}$ be the partial character table of a finite group $G$, where the column $d$ corresponding to $g\in G$ is missing.
We need to show that $d$ is uniquely determined by $C$. Since character tables are invertible, the columns of $C$ span a vector space of dimension $k-1$. 
By the second orthogonality relation, $d$ spans the orthogonal complement of this space. In particular, $d$ is uniquely determined up to a scalar multiple. If $C$ has only one row of the form $(1,\ldots,1)$, then this row must correspond to the trivial character. In this case $d$ is uniquely determined. 

Now assume that $C$ has two rows $(1,\ldots,1)$. Then there exists a non-trivial character $\chi\in\Irr(G)$ such that $\chi(h)=1$ for all $h$ outside the conjugacy class of $g$. Let $K:=\Ker(\chi)<G$. Then $G\setminus K$ is the conjugacy class of $g$. It follows that $|G:K|=2$ and $|\C_G(g)|=2$. Moreover, $\chi(g)=-1$. Since $\chi(1)^2+\chi(g)^2=2=|\C_G(g)|$, the orthogonality relation implies that $\psi(g)=0$ for all $\psi\in\Irr(G)\setminus\{1_G,\chi\}$. Hence, $d=(1,-1,0,\ldots,0)^\TT$ up to permutation of rows. 
\end{proof}

Before we embark with the corresponding theorem for rows, we illustrate why it must lie deeper. The following matrices differ only by the their last row:
\[
\begin{pmatrix}
1&1&1&1&1&1\\
1&1&1&1&-1&-1\\
1&1&1&-1&1&-1\\
1&1&1&-1&-1&1\\
2&2&-2&0&0&0\\
8&-1&0&0&0&0
\end{pmatrix}\qquad
\begin{pmatrix}
1&1&1&1&1&1\\
1&1&1&1&-1&-1\\
1&1&1&-1&1&-1\\
1&1&1&-1&-1&1\\
2&2&-2&0&0&0\\
4&-2&0&0&0&0
\end{pmatrix}
\]
The first matrix is the character table of the sharply $2$-transitive Mathieu group $M_9\cong C_3^2\rtimes Q_8$. The second fulfills the orthogonality relations and looks like the character table of a group of the form $C_3\rtimes Q_8$ or $C_3\rtimes D_8$ (cf. \cite[Definition~2.2, Theorem~2.3]{GagolaFormal}). However, using Clifford theory one can show that there is no group with this character table (the details become apparent below). This gives rise to a pseudo group mentioned in the introduction.

Nevertheless, the proof of the following “row theorem” features some duality to \autoref{thmcols}. 

\begin{Thm}\label{thmrows}
There are no finite groups whose character tables differ by only one row.
\end{Thm}
\begin{proof}
Let $C=(c_{ij})\in\CC^{(k-1)\times k}$ be the partial character table of $G$, where the row corresponds to $\chi\in\Irr(G)$ is missing. It suffices to show that $\chi$ is uniquely determined by $C$. 
We may assume that $\chi$ is real, since otherwise $\chi$ is complex conjugate to a given character. Moreover, we can assume that $\chi$ is not the trivial character. In particular, $G\ne 1$.
We first identity which column of $C$ corresponds to the trivial element. This must be a column filled with positive integers. Using that $|\psi(g)|\le\psi(1)$ for every $g\in G$ and $\psi\in\Irr(G)$, the trivial element corresponds to an integral column with “maximal” entries. If there are more than one such identical columns (see Case~2 below), we pick one of them and assign it to the trivial element.

Let $1=g_1,\ldots,g_k\in G$ be representatives for the conjugacy classes of $G$. Let $d:=\chi(1)$ and \[\gamma_{st}:=\sum_{i=1}^{k-1}c_{is}\overline{c_{it}}\qquad (1\le s,t\le k).\] 
By the second orthogonality relation, $\chi(g_s)=-\frac{1}{d}\gamma_{1s}$ for $s=2,\ldots,k$. Hence, $\chi$ is uniquely determined by $d$. 

\textbf{Case~1:} There exist $1<s<t$ such that $\gamma_{st}\ne 0$.\\
Since $\chi(g_s)\chi(g_t)=-\gamma_{st}$, we have $\chi(g_s)\ne 0\ne\chi(g_t)$. It follows that
\[d^2=\frac{\gamma_{1s}\gamma_{1t}}{\chi(g_s)\chi(g_t)}=-\frac{\gamma_{1s}\gamma_{1t}}{\gamma_{st}},\]
and $d$ is uniquely determined by $C$ (note that $d>0$). 

\textbf{Case~2:} $\gamma_{st}=0$ for all $1<s<t$.\\
Since $\chi$ cannot vanish identically on $G\setminus\{1\}$ (otherwise $[\chi,1_G]\ne 0$), there exists $r>1$ such that $\chi(1)\chi(g_r)=\gamma_{1r}\ne 0$. Without loss of generality, let $r=2$. Then $\gamma_{2s}=0$ implies $\chi(g_s)=0$ for $s=3,\ldots,k$. This situation was studied by Gagola~\cite{Gagola}. We repeat some of his arguments for the convenience of the reader. 
For every $\psi\in\Irr(G)\setminus\{\chi\}$ we have
\[0=[\psi(1)1_G-\psi,\chi]=\frac{1}{|G|}\bigl(\psi(1)-\psi(g_2)\bigr)\chi(g_2)\]
since $\chi$ is non-trivial. This yields $\psi(g_2)=\psi(1)$, i.\,e. the first two columns of $C$ are identical.
Moreover, $g:=g_2$ is contained in the kernel of every $\psi\in\Irr(G)\setminus\{\chi\}$. 
Since $\gamma_{1s}=0$ for $s\ge 3$, none of the other columns of $C$ is identical to the first column.
This shows that
\[N:=g^G\cup\{1\}=\bigcap_{\psi\in\Irr(G)\setminus\{\chi\}}\Ker(\psi)\unlhd G\]
is a minimal normal subgroup of $G$. Since all non-trivial elements of $N$ are conjugate, $N$ must be an elementary abelian $p$-group.  By Clifford theory, $\chi$ is the only irreducible character of $G$ lying over some $\lambda\in\Irr(N)\setminus\{1_N\}$ and $\lambda$ is $G$-conjugate to all non-trivial characters of $N$. It follows that $\lambda$ is fully ramified in its stabilizer $G_\lambda$ and \[|G:G_\lambda|=|\Irr(N)|-1=|N|-1.\] 
Moreover, $\chi=\frac{1}{e}\lambda^G$ where $e$ is the ramification index of $\lambda$ (see \cite[Lemma~8.2]{Navarro2}). 
Note that $|G/N|=\sum_{\psi\ne\chi}\psi(1)^2=\gamma_{11}$ is determined by $C$. 
For a prime $q\ne p$, $\chi$ vanishes on the $q$-singular elements. This means that $\chi$ has $q$-defect $0$ and 
\[d_{p'}=|G|_{p'}=|G/N|_{p'}\]
(see \cite[Corollary~4.7]{Navarro2}). 
From $|G|=\gamma_{11}+d^2=|G/N|+d^2$ we also get $|N|=1+\frac{d^2}{|G/N|}$ and 
\[d_p^2=|G/N|_p.\] 
We have thus shown that $\chi$ is determined by $p$ alone. Notice further that 
\[|G_\lambda/N|=e^2=\Bigl(\frac{\lambda^G(1)}{\chi(1)}\Bigr)^2=\frac{|G/N|^2}{d^2}=|G/N|_p,\]
i.\,e. $G_\lambda/N$ is a Sylow $p$-subgroup of $G/N$.

Removing the second column of $C$ reveals the character table of $G/N$.
The following property can be read off from this character table (see \cite[Corollary~3.12]{Navarro2}).

\textbf{Case~2.1:} $G/N$ has a non-cyclic Sylow $q$-subgroup of some odd prime $q$.\\
Suppose that $G_\lambda=N$. Then 
\[|G:N|=|G:G_\lambda|=|N|-1=|G:\C_G(g)|\]
yields $N=\C_G(g)$. Consequently, $G$ is a Frobenius group with kernel $N$ and complement isomorphic to $G/N$ (see \cite[Theorem~6.7]{IsaacsGroup}). However, it is well-known that the Sylow subgroups of a Frobenius complement are cyclic or quaternion groups (see \cite[Theorem~6.11]{IsaacsGroup}). This contradiction shows that $N<G_\lambda$ and $e>1$. 
A theorem attributed to Gagola and presented with an elementary (but long) proof by Isaacs~\cite[Theorem~5.1]{IsaacsFR} states that 
\[|G/N|_{p'}<|N|<|G_\lambda:N|=|G/N|_p.\]
For every prime $q\ne p$, we have $|G/N|_q\le|G/N|_{p'}<|G/N|_p\le|G/N|_{q'}$. Hence, $p$ is uniquely determined by $\gamma_{11}=|G/N|$.

For the remainder of the proof we assume that all Sylow subgroups of $G/N$ of odd order are cyclic. If $|G_\lambda/N|=|G/N|_p\ne 1$, then $G_\lambda/N$ cannot be cyclic, as otherwise $\lambda$ would extend to $G_\lambda$. Hence, it suffices to distinguish $p=2$ from $|G/N|_p=1$. In the latter, case $d=|G/N|$ is independent of $p$. If a Sylow $2$-subgroup $P/N$ of $G/N$ is cyclic, then clearly $|G/N|_p=1$. Thus, we assume that $P/N$ is not cyclic. The next case can also be read off from the character table of $G/N$ by \cite[Theorem~A]{NaSaTi}. 

\textbf{Case~2.2:} $|P/N:(P/N)'|>4$.\\
Assuming $G_\lambda=N$, we end up with a Frobenius group as in Case~2.1. But then $P/N$ must be a quaternion group with $|P/N:(P/N)'|=4$. This contradiction shows that $p=2$.

\textbf{Case~2.3:} $|P/N:(P/N)'|=4$.\\
Here, $P/N$ has a cyclic subgroup $Q/N$ of index $2$. 
If $|P/N|=4$, then again $G/N$ cannot be isomorphic to a Frobenius complement and we have $p=2$. Hence, let $|P/N|\ge 8$. 
If $p=2$, then we may assume that $G_\lambda=P$ by Sylow's theorem. Now $\lambda$ extends to $Q$. This implies $e=2$ and $|P/N|_2=e^2=4$, against our assumption. Therefore, $p\ne 2$, $|G/N|_p=1$ and $d=|G/N|$. This completes the proof.
\end{proof}

The two abelian groups of order $4$ show that character tables can differ by only two columns, two rows or by just $k$ entries, where $k$ is the total number of characters. It might be possible to reconstruct a row and a column of a partial character table simultaneously, but this seems to require an analysis of characters vanishing on all but three conjugacy classes. For instance, the reader may try to decide which of the following matrices are character tables (here $\ii=\sqrt{-1}$):
\[
\begin{pmatrix}
1&1&1&1&1&1\\
1&1&-1&-1&1&1\\
1&-1&\ii&-\ii&1&1\\
1&-1&-\ii&\ii&1&1\\
4&0&0&0&-2&0\\
2&0&0&0&2&-2
\end{pmatrix}
\quad
\begin{pmatrix}
1&1&1&1&1&1\\
1&1&-1&-1&1&1\\
1&-1&\ii&-\ii&1&1\\
1&-1&-\ii&\ii&1&1\\
4&0&0&0&-2&1\\
4&0&0&0&1&-2
\end{pmatrix}
\quad
\begin{pmatrix}
1&1&1&1&1&1\\
1&1&-1&-1&1&1\\
1&-1&\ii&-\ii&1&1\\
1&-1&-\ii&\ii&1&1\\
4&0&0&0&-3&1\\
8&0&0&0&1&-1
\end{pmatrix}
\]

The proofs of our theorems in combination with \cite[Corollay~3.12]{Navarro2} and \cite[Theorem~A]{NaSaTi} provide a practical algorithm to complete a partial character table. 
We challenge the reader to add three rows and three columns to turn the following matrix into a character table of size $11\times 11$. There is only one way to do this, but two non-isomorphic groups share this character table:

\[
\begin{pmatrix}
1&1&1&1&1&1&1&1\\
1&1&1&1&1&1&-1&-1\\
1&-1&-1&1&1&-1&\ii&\ii\\
2&-2&-2&2&-1&1&0&0\\
2&2&2&2&-1&-1&0&0\\
3&3&-1&-1&0&0&-1&1\\
3&3&-1&-1&0&0&1&-1\\
3&-3&1&-1&0&0&\ii&-\ii
\end{pmatrix}
\]

\section{Large character degrees}
The reader might have noticed that the difficulties in our proofs arise from groups with “large” character degrees. 
Gagola's bound $|N|<|G/N|_p$, used in \autoref{thmrows}, has been improved in \cite[Theorem~1.2]{HLS} as follows: 

\begin{Thm}
Let $G$ be a group of order $d(d+e)$ where $d$ is the degree of an irreducible character and $e>1$ is an integer. Then $|G|\le e^4-e^3$.
\end{Thm}
 
While \cite{HLS} depends on the classification of the finite simple groups, our proof of \autoref{thmrows} (relying on \cite[Proposition~2.1]{NaSaTi}) is CFSG-free. 
Due to a construction by Isaacs~\cite{IsaacsFR}, the bound $|G|\le e^4-e^3$ is best possible whenever $e$ is a power of a prime.
The authors of \cite{HLS} have asked to classify those groups. 
Building on work of Snyder~\cite{Snyder} for $e=2,3$, Durfee--Jensen~\cite{DurfeeJensen} have classified the groups with $2\le e\le 6$ (there are infinitely many groups for $e=1$). For $e=7$, they could not finish their classification since the groups of order $d(d+e)=42\cdot 49=2958$ are not available in the small groups library. However, these groups can be constructed using the \texttt{GrpConst} package in GAP~\cite{GAPnew} (see also \cite{HornHomepage}). There are just four of them with an irreducible character of degree $d=42$. 
We extend the classification to $e\le 11$. Most group orders can be handled with GAP. The difficult cases, which require special attention, are settled in the following lemmas.

\begin{Lem}
Let $G$ be a group of order $d(d+e)$ with $\chi\in\Irr(G)$ of degree $d$. Then $(d,e)$ is not one of the following pairs:
\begin{multicols}{4}
\begin{enumerate}[(i)]
\item $(32,8)$.
\item $(48,8)$.
\item $(54,9)$.
\item $(55,9)$.
\item $(54,10)$.
\item $(80,10)$.
\item $(64,11)$.
\end{enumerate}
\end{multicols}
\end{Lem}
\begin{proof}
First we recall some general facts from Clifford theory. Let $N\unlhd G$. Let $\theta\in\Irr(N)$ be a constituent of $\chi_N$ with ramification index $e$ and $k:=|G:G_\theta|$, where $G_\theta$ is the stabilizer of $\theta$. Then $\chi(1)=ke\theta(1)$ such that $e\mid|G_\theta:N|$ and $e^2\le|G_\theta:N|$ (see \cite[Theorem~5.12 and the subsequent remark]{Navarro2}). Moreover, $k\theta(1)^2<|N|$ and \[\chi(1)^2\le|G:N|\frac{|N|-1}{\theta(1)^2}\] 
unless $N=1$. If $N$ is abelian, then $\chi(1)=ke\mid|G:N|$ and if $N\le\Z(G)$, then $\chi(1)=e$. 

\begin{enumerate}[(i)]
\item Here $|G|=2^8\cdot 5$. Recall that $N:=\pcore_2(G)$ is the kernel of the transitive action of $G$ on the cosets of a Sylow $2$-subgroup. In particular, $G/N$ is isomorphic to a subgroup of $S_5$. It follows that $|N|\ge 2^6$. But this leads to the contradiction
\[2^5=\chi(1)=ke\theta(1)<|G:N|_2\sqrt{|N|}=\frac{2^8}{\sqrt{|N|}}\le 2^5.\]

\item Here $|G|=2^7\cdot 3\cdot 7$. By the classification of the transitive groups of degree $21$ in GAP, we obtain $\pcore_2(G)\ne 1$. Hence, there exists a minimal normal subgroup $N\le\pcore_2(G)$. Since $N$ is (elementary) abelian and $48=\chi(1)\mid|G:N|$, it follows that $|N|\le 8$. On the other hand, we have $\chi(1)^2\le|G|-|G:N|$, which forces $|N|=8$. Since $k$ divides $\chi(1)=48$ and $k<|N|$, we get $k\le 6$. But now 
\[2^8\cdot 3^2=\chi(1)^2\le|G:N|k=2^5\cdot 3^2\cdot 7,\] 
a contradiction.

\item Here $|G|=2\cdot 3^5\cdot 7$. There exists a normal subgroup of order $3^5\cdot 7$, and by Sylow's theorem $G$ has a normal Sylow $7$-subgroup $N$. Then $\chi(1)^2=2^2\cdot 3^6=|G|-|G:N|$ and therefore $k=6$ and $e=9$. This is impossible since $\theta$ extends to $G_\theta$ by \cite[Corollary~6.2]{Navarro2}.

\item Here $|G|=55\cdot 64=2^6\cdot 5\cdot 11$. Then $\chi$ has $11$-defect $0$ and therefore $\pcore_{11}(G)=1$. By Sylow's theorem, $G$ has $2^6\cdot 5$ Sylow $11$-subgroups. Hence, $G$ is a Frobenius group with kernel $K$ of order $2^6\cdot 5$. By Thompson's theorem on Frobenius kernels, $K$ is nilpotent. Therefore, $K$ and $G$ have a normal Sylow $5$-subgroup. But this contradicts the fact that $\chi$ has $5$-defect $0$. 

\item Here $|G|=2^7\cdot 3^3$. Then $\chi$ has $3$-defect $0$ and therefore $\pcore_3(G)=1$. Let $N:=\pcore_2(G)$. By the Hall--Higman lemma, $\C_G(N)\le N$ (see \cite[Theorem~3.21]{IsaacsGroup}). Let $P$ be a Sylow $3$-subgroup of $G$. 
Since $P$ acts faithfully on $N/\Phi(N)$, we obtain $|N|\ge |N/\Phi(N)|\ge 2^6$. Suppose that $|N|=2^6$. Then $N$ is elementary abelian and $G_\theta=N$. In particular, $P$ has a regular orbit on $\Irr(N)$. By \cite[Corollary~2.12]{Navarro2}, $P$ also has a regular orbit on $N$. Using the local structure of $\GL(6,2)$, one can show with GAP that this is impossible. Therefore, $|N|=2^7$ and $G=N\rtimes P$. We can now determine the candidates for $N$ with the small groups library (there are four such groups with an automorphism group of order divisible by $27$). For each candidate we construct $G$ and check that $\chi$ does not exist. 

\item Here $|G|=2^5\cdot 3^2\cdot 5^2$. Since $\chi$ has $2$-defect $1$, it must lie in a $2$-block of defect $1$. In particular,
$N:=\pcore_2(G)$ has order at most $1$. If $N\ne 1$, then we derive the contradiction 
\[2^8\cdot 5^2=\chi(1)^2\le|G:N|(|N|-1)=2^4\cdot 3^2\cdot 5^2.\]
Hence, $N=1$. In the same way we can show that $\pcore_5(G)=1$.
If $G$ is solvable, we must have $N:=\pcore_3(G)\ne 1$ and $\C_G(N)\le N$ by the Hall--Higman lemma. This cannot happen since $G/\C_G(N)\le\GL(3,2)$ is too small. Thus, $G$ is non-solvable. Let $N$ be a minimal normal subgroup of $G$. 
If $N\cong A_6$, then $\C_G(N)\ne1$ is solvable normal subgroup of $G$ since $N\cap\C_G(N)=\Z(N)=1$ and $G/N\C_G(N)\le\Out(A_6)\cong C_2^2$. This contradicts $\pcore_5(G)=1$. If $N\cong A_5^2$, then $N$ must contain an irreducible character of degree $80$ or $40$, because $|G:N|=2$. This is not the case since $A_5$ has character degrees $1,3,4,5$. Finally, let $N\cong A_5$. 
If $M:=\C_G(N)$ is solvable, we get the contradiction $\C_M(\pcore_3(M))\le M$ as above. If $M$ is non-solvable, then we find another normal subgroup $M\unlhd G$ isomorphic to $A_5^2$. This is impossible as we have just seen.

\item Here $|G|=2^6\cdot 3\cdot 5^2$. Since $\chi$ has $2$-defect $0$, we have $\pcore_2(G)=1$. 
As in the previous case, we can show that $G$ has to be non-solvable. 
Let $N$ be a non-abelian minimal normal subgroup. Then $N\cong A_5$ and $\C_G(N)$ contains an abelian minimal normal subgroup of $G$. This leads to a contradiction as before.
\qedhere
\end{enumerate}
\end{proof}

\begin{Lem}
There are exactly $12$ groups of order $3584$ with an irreducible character of degree $56$.
\end{Lem}
\begin{proof}
Let $G$ be a group of order $3584=2^9\cdot 7$ with $\chi\in\Irr(G)$ of degree $56$. Let $N:=\pcore_2(G)$. By the classification of transitive subgroups of $S_7$, we have $|N|\ge 2^8$. Suppose first that $|N|=2^8$. Let $M/N=\pcore_7(G/N)$. Then $|M|=2^8\cdot 7$ and $\chi$ lies over some $\theta\in\Irr(M)$ of degree $28$. A GAP computation shows that there are 11 possible isomorphism types for $N$. 
However, in each case $G/N\cong D_{14}$ is not isomorphic to a subgroup of $\Out(N)$. This contradiction shows that $|N|=2^9$ and $G\cong N\rtimes C_7$. 

Now let $\theta\in\Irr(N)$ and $\lambda\in\Irr(\Z(N))$ be constituents of $\chi_N$ and $\chi_{\Z(N)}$ respectively. If $G/N$ acts trivially on $\Z(N)$, then $\lambda$ is $G$-invariant and we derive the contradiction 
\[56^2\le|G:\Z(N)|\le|G|/2=56\cdot 32.\] 
Hence, $G/N$ acts faithfully on $\Z(N)$ and it follows that $|\Z(N)|\ge8$. On the other hand, $64=\theta(1)^2\le|N:\Z(N)|\le 2^6$ implies $|\Z(N)|=8$. 
Since $G/N$ acts irreducibly on $\Z(N)$, $\Z(N)$ is a minimal normal subgroup of $G$. In particular, $\Z(N)$ is elementary abelian and $\Z(N)\subseteq N'$. 
We use GAP to enumerate the groups of order $2^6$ with an automorphism of order $7$. In this way we find just 7 possibilities for $N/\Z(N)$. With the notation of \cite[Definition~2.1]{OBrienAnu}, $N$ is an \emph{immediate descendant} of $N/\Z(N)$ and those can be computed with the \texttt{AnuPQ} package in GAP. It turns out that $N/\Z(N)$ must be elementary abelian. Hence, $\Z(N)=N'=\Phi(N)$. In particular, $N$ has rank $6$ and $p$-class $2$. According to the small groups library, those groups have the form $N=\mathtt{SmallGroup}(2^9,a)$ with $7{,}532{,}393\le a\le 10{,}481{,}221$. Running through these groups with GAP yields the following values for $a$:
\[10475413,\,10476872,\,10477010,\,10477017,\,10481182,\,10481184,\,10481185,\,10481201,\,10481221\]
(we made use of the \texttt{AutPGrp} package to compute $\Aut(N)$).
If $a\ne 10481201$, then $|\Aut(N)|_7=7$ and there is a unique group $G$. If $a=10481201$, then $|\Aut(N)|_7=7^2$ and there are just four non-isomorphic groups $G$. 
\end{proof}

The final theorem summarizes our findings.

\begin{Thm}
Let $k$ be the number of non-isomorphic groups $G$ of order $n=d(d+e)$ where $2\le e\le 11$ and $G$ has an irreducible character of degree $d$. Then $(d,n)^k$ is given in the table below:
\[
\begin{array}{c|l|l}
e&(d,n)^k&\Sigma\\\hline
2&(1,3),(2,8)^2&3\\
3&(1,4)^2,(2,10),(6,54)^2&5\\
4&(1,5),(2,12)^2,(3,21),(4,32)^7,(12,192)^6&17\\
5&(1,6)^2,(2,14),(3,24)^3,(4,36)^2,(20,500)^3&11\\
6&(1,7),(2,16)^9,(3,27)^2,(4,40)^2,(5,55),(6,72)^3&18\\
7&(1,8)^5,(2,18)^3,(5,60),(6,78),(8,120),(9,144),(42,2058)^4&16\\
8&(1,9)^2,(2,20)^2,(4,48)^{10},(6,84)^2,(8,128)^{75},(12,240)^2,(24,768)^{11},(56,3584)^{12}&116\\
9&(1,10)^2,(2,22),(3,36)^2,(4,52),(7,112),(8,136),(12,252),&42\\
&(16,400)^2,(18,486)^{13},(72,5832)^{18}\\
10&(1,11),(2,24)^{11},(3,39),(6,96)^{12},(8,144)^5,(9,171),(14,336),(18,504)^2&34\\
11&(1,12)^5,(2,26),(3,42),(4,60)^4,(5,80),(16,432)^5,(21,672)^2,(24,840),(110,13310)^6&26
\end{array}
\]
\end{Thm}

\section*{Acknowledgment}
I thank Gabriel Navarro for some useful comments on this paper. Alexander Hulpke has promptly fixed some bugs in GAP, which were discovered in the course of this work.
The work is supported by the German Research Foundation (\mbox{SA 2864/4-1}).


\begin{thebibliography}{10}

\bibitem{BrauerLectures}
R. Brauer, \textit{Representations of finite groups}, in: Lectures on {M}odern
  {M}athematics, {V}ol. {I}, 133--175, Wiley, New York, 1963.

\bibitem{BrauerPseudo}
R. Brauer, \textit{On pseudo groups}, J. Math. Soc. Japan \textbf{20} (1968),
  13--22.

\bibitem{DurfeeJensen}
C. Durfee and S. Jensen, \textit{A bound on the order of a group having a large
  character degree}, J. Algebra \textbf{338} (2011), 197--206.

\bibitem{Gagola}
S.~M. Gagola, \textit{Characters vanishing on all but two conjugacy classes},
  Pacific J. Math. \textbf{109} (1983), 363--385.

\bibitem{GagolaFormal}
S.~M. Gagola, \textit{Formal character tables}, Michigan Math. J. \textbf{33}
  (1986), 3--10.

\bibitem{GAPnew}
The GAP~Group, \textit{GAP -- Groups, Algorithms, and Programming, Version
  4.12.0}; 2022, (\url{http://www.gap-system.org}).

\bibitem{HarrisPseudo}
M.~E. Harris, \textit{A note on pseudo groups}, J. Fac. Sci. Univ. Tokyo Sect.
  I \textbf{16} (1969), 255--272.

\bibitem{HornHomepage}
M. Horn, \textit{Numbers of isomorphism types of finite groups of given order},
  \url{https://groups.quendi.de/}.

\bibitem{HLS}
N.~N. Hung, M.~L. Lewis and A.~A. Schaeffer~Fry, \textit{Finite groups with an
  irreducible character of large degree}, Manuscripta Math. \textbf{149}
  (2016), 523--546.

\bibitem{IsaacsGroup}
I.~M. Isaacs, \textit{Finite group theory}, Graduate Studies in Mathematics,
  Vol. 92, American Mathematical Society, Providence, RI, 2008.

\bibitem{IsaacsFR}
I.~M. Isaacs, \textit{Bounding the order of a group with a large character
  degree}, J. Algebra \textbf{348} (2011), 264--275.

\bibitem{Navarro2}
G. Navarro, \textit{Character theory and the {M}c{K}ay conjecture}, Cambridge
  Studies in Advanced Mathematics, Vol. 175, Cambridge University Press,
  Cambridge, 2018.

\bibitem{NaSaTi}
G. Navarro, B. Sambale and P.~H. Tiep, \textit{Characters and {S}ylow
  2-subgroups of maximal class revisited}, J. Pure Appl. Algebra \textbf{222}
  (2018), 3721--3732.

\bibitem{OBrienAnu}
E.~A. O'Brien, \textit{The {$p$}-group generation algorithm}, 677--698, Vol. 9,
  1990.

\bibitem{Snyder}
N. Snyder, \textit{Groups with a character of large degree}, Proc. Amer. Math.
  Soc. \textbf{136} (2008), 1893--1903.

\end{thebibliography}
\end{document}